\documentclass[12pt]{amsart}

\usepackage{graphicx}
\usepackage{amsthm}
\usepackage{amsmath, amssymb}
\swapnumbers

\theoremstyle{plain}
\newtheorem{Thm}{Theorem}

\newtheorem{Coro}[Thm]{Corollary}
\newtheorem{Lem}[Thm]{Lemma}

\theoremstyle{definition}
\newtheorem{Def}[Thm]{Definition}

\newcommand{\mS}{\mathcal{S}}

\begin{document}

\title[An upper bound on common stabilizations]{An upper bound on common stabilizations of Heegaard splittings}

\author{Jesse Johnson}
\address{\hskip-\parindent
        Department of Mathematics \\
        Oklahoma State University \\
        Stillwater, OK 74078 \\
        USA}
\email{jjohnson@math.okstate.edu}


\thanks{This project was supported by NSF Grant DMS-1006369}

\begin{abstract}
We show that for any two Heegaard splittings of genus $p$ and $q$ for the same closed 3-manifold, there is a common stabilization of genus at most $\frac{3}{2}p + 2q - 1$.  One may compare this to recent examples of Heegaard splittings whose smallest common stabilizations have genus at least $p+q$ or $p + \frac{1}{2} q$ depending on the notion of equivalence.
\end{abstract}

\maketitle

A \textit{Heegaard splitting} of a compact, closed, orientable 3-manifold $M$ is a triple $(\Sigma, H^-_\Sigma, H^+_\Sigma)$ where $\Sigma$ is a compact, closed, separating surface in $M$ and $H^-_\Sigma$, $H^+_\Sigma$ are embedded handlebodies in $M$ such that $\partial H^-_\Sigma = \Sigma = \partial H^+_\Sigma = H^-_\Sigma \cap H^+_\Sigma$.  There are two notions of isotopy equivalence for Heegaard splittings that one can consider.  Under \textit{unoriented isotopy}, two Heegaard splittings are considered equivalent if there is an isotopy that takes one surface onto the other.  Under the stricter notion of \textit{oriented isotopy}, we also require that the isotopy take the first handlebody in one triple to the first handlebody in the other.  In particular $(\Sigma, H^-_\Sigma, H^+_\Sigma)$ and $(\Sigma, H^+_\Sigma, H^-_\Sigma)$ will be equivalent under unoriented isotopy, but may be distinct under oriented isotopy.  For either notion of equivalence, every 3-manifold will contain many different isotopy classes of Heegaard splittings.

In particular, a new Heegaard splitting can always be constructed from a given splitting by drilling one or more unknotted holes out of one of the handlebodies and attaching handles to the other handlebody that pass through these holes.  This new splitting is called a \textit{stabilization} and Reidemeister~\cite{reid} and Singer~\cite{sing} showed independently that given any two Heegaard splittings for the same 3-manifold, there is a third Heegaard splitting that is isotopic to a stabilization of each the original two.

The original proofs of this fact do not suggest what the genus of the common stabilization should be in terms of the original genera.  Rubinstein and Scharlemann~\cite{rub:compar} proved that if $M$ is non-Haken, then there is a common stabilization of genus at most $5p + 8q - 9$, where $p$ and $q$ are the genera of the original Heegaard splittings and $p \geq q$.  They later found a quadratic bound for Haken manifolds.  (Both bounds are valid for oriented or unoriented isotopy.)  

This upper bound has long been believed to be higher than necessary.  Recently, examples have been found by Hass-Thompson-Thurston~\cite{htt:stabs} of pairs of Heegaard splittings whose smallest common stabilization, up to oriented isotopy, has genus $p + q$.  For unoriented isotopy, the author of the present paper~\cite{me:stabs2} found pairs of Heegaard splittings whose smallest common stabilization is just below $p + \frac{1}{2} q$.  Kazuto Takao~\cite{takao} improved the methods in~\cite{me:stabs2} to show that there are Heegaard splittings of certain connect sum 3-manifolds whose smallest common stabilization has genus exactly $p + \frac{1}{2}q$. (In these last examples, $p = q$ and both are even).  For both types of equivalence, Dave Bachman~\cite{bachman} found examples with slightly lower stable genera than those mentioned above, around the same time.

In the present paper, we narrow the gap between the upper bound and the known examples.

\begin{Thm}
\label{mainthm}
If $(\Sigma, H^-_\Sigma, H^+_\Sigma)$, $(R, H^-_R, H^+_R)$ are Heegaard splittings for a 3-manifold $M$ such that the genera of $\Sigma$ and $R$ are $p$ and $q$, respectively, then there is a third Heegaard splitting $(T, H^-_T, H^+_T)$ of genus at most $\frac{3}{2}p + 2q - 1$ such that $(T, H^-_T, H^+_T)$ is isotopic to a stabilization of $(\Sigma, H^-_\Sigma, H^+_\Sigma)$ and to a stabilization of $(R, H^-_R, H^+_R)$.
\end{Thm}

This result is for oriented or unoriented isotopy.  The proof uses a result of Fengchun Lei~\cite{lei}, to turn a nice position of a spine for one Heegaard splitting with respect to a sweep-out for the other into a common stabilization.  Lei's result is described in Section~\ref{connsect}.

The bulk of the paper is devoted to finding this nice position.  Specifically, we define a type of thin position for surfaces with respect to a pair of embedded handlebodies, from which we derive a structure very similar to the Rubinstein-Scharlemann graphic for a pair of sweep-outs~\cite{rub:compar}.  The thin position setting allows us to get much more control of the behaviour of the graphic.  

We define this new type of thin position using the axiomatic method introduced in an earlier paper~\cite{axiomatic}. We review the axiomatic setup and define handlebody thin position in Sections~\ref{hcomsect} and~\ref{thinaxiomsect}.  All but two of the axioms are immediate for handlebody thin position and this is proved in Section~\ref{thinhandlesect}.  We prove that the final two axioms hold in Sections~\ref{boundaryaxiomsect} and~\ref{cassongordonsect}.

We will be interested in the case when the complement of the two handlebodies is a surface-cross-interval.  In Section~\ref{flatsect}, we define flat surfaces, as a combinatorial model for surfaces in this structure.  Section~\ref{esssurfsect} and~\ref{bandmovesect} define essential flat surfaces and describe a type of combinatorial move that can be used to modify flat surfaces.  In Section~\ref{indexonesect}, we show that every index-one surface with respect to such a pair of handlebodies can be made essential, then in Section~\ref{essisosect} we describe how a thin path in the complex of surfaces can be a realized by a sequence of essential flat surfaces related by the moves defined in Section~\ref{bandmovesect}.  Finally, in Section~\ref{mainthmsect}, we apply these results to prove Theorem~\ref{mainthm}.

I thank Alex Coward, Joel Hass, Martin Scharlemann and Abby Thompson for many helpful suggestions.

\section{Common Stabilizations}
\label{connsect}

One of the main tools in this paper is a result that is proved, though not stated explicitly in a paper of Fengchun Lei.  It follows from the classification of Heegaard splittings of compression bodies, which is a corollary of Scharlemann and Thompson's classification~\cite{sct:crossi} of Heegaard splittings of any surface cross an interval.

A handlebody $H$ is by definition homeomorphic to a regular neighborhood of a graph $K \subset H$ such that $H$ deformation retracts onto $K$.  Any graph in $H$ with this property is called a \textit{spine} of $H$.

\begin{Lem}[Lei~\cite{lei}]
\label{leilem}
If $(\Sigma, H^-_\Sigma, H^+_\Sigma)$ and $(T, H^-_T, H^+_T)$ are Heegaard splittings of a 3-manifold $M$ such that a spine $K^-_\Sigma$ of $H^-_\Sigma$ is contained in a spine $K^-_T$ of $H^-_T$ then $(T, H^-_T, H^+_T)$ is a stabilization of $(\Sigma, H^-_\Sigma, H^+_\Sigma)$.
\end{Lem}

It follows from this Lemma that to find a common stabilization of two Heegaard splittings, we need to find a graph that determines a Heegaard splitting and contains spines for handlebodies in the initial two Heegaard splittings.  In order to do this, we must put the two spines in a simple position relative to each other as follows:

A \textit{sweep-out} of a Heegaard splitting is a smooth function $f : M \rightarrow [-1,1]$ such that $f^{-1}(-1)$ is a spine for $H^-_\Sigma$, $f^{-1}(1)$ is a spine for $H^+_\Sigma$ and there are no critical points away from these graphs.  In such a function, the level sets $\Sigma_t = f^{-1}(t)$ for $t \in (-1,1)$ will be pairwise disjoint, embedded surfaces parallel to $\Sigma$.

Consider a spine $K = K^+_R$ of $H^+_R$ that is disjoint from the spines $K^-_\Sigma$, $K^+_\Sigma$.  Each component of intersection $\Sigma_t \cap K$ will be called a \textit{horizontal component}. We are not assuming that $K$ is transverse to the surfaces $\Sigma_t$, so horizontal components may contain non-trivial subgraphs of $K$. A horizontal component $X$ will be called \textit{locally maximal} if $X$ is also a connected component of $K \cap f^{-1}([t,1])$.  

\begin{Lem}
\label{esspinelem}
Let $(\Sigma, H^-_\Sigma, H^+_\Sigma)$, $(R, H^-_R, H^+_R)$ be Heegaard splittings for $M$ with genera $p$, $q$, respectively.  If $K^+_R$ is a spine for $H^+_R$ and there are  $n$ locally maximal horizontal components then there is a common stabilization for $\Sigma$, $R$ of genus at most $p + q + n - 1$.
\end{Lem}

\begin{proof}
Given a sweep-out $f$ for $\Sigma$, let  $K^+_\Sigma = f^{-1}(1)$ be the spine for $H^+_\Sigma$ defined by $f$.  Assume we have chosen the spine $K^+_R$ for $H^+_R$ so that every horizontal level is either a single point or a graph with a single vertex.  (We can do this by collapsing any horizontal edge with distinct endpoints down to a single vertex.) Because $f$ is a sweep-out, we can choose a vertical arc (with respect to $f$) from each locally maximal component of $K^+_R$ to the graph $K^+_\Sigma$.  Let $K^+_T$ be the union of $K^+_\Sigma$, $K^+_R$ and this collection of vertical arcs.  Because $K^+_R$ and $K^+_\Sigma$ are disjoint, the genus of the graph is $p + q + n - 1$, where $p$ and $q$ are the genera of $K^+_R$ and $K^+_\Sigma$, and they are connected by $n$ vertical arcs.  We will show that the graph $K^+_T$ defines a Heegaard splitting $(T, H^-_T, H^+_T)$.  By Lemma~\ref{leilem}, this Heegaard splitting will be a common stabilization of genus $p + q + n - 1$.

Note that for the graph $K^+_T$, the only locally maximal horizontal component is $K^+_\Sigma$.  Thus if $v$ is the highest vertex in $K^+_T$ outside of $K^+_\Sigma$, there will be a path $\alpha$ from $v$ to a point in $K^+_\Sigma$ that is non-decreasing with respect to $f$.  Because $v$ is maximal, the path $\alpha$ cannot have any vertices in its interior.  This implies that $\alpha$ follows a single monotonic edge.  Let $K_1$ be the result of shrinking this edge to a point, pulling $v$ up into $K^+_\Sigma$ and extending all the other edges with endpoint in $v$ up along $\alpha$.  This proceedure is illustrated in Figure~\ref{commonspinefig}.  The new graph still contains $K^+_\Sigma$ and has the property that $K^+_\Sigma$ is the only maximal horizontal component.  Moreover, a regular neighrborhood of $K_1$ is isotopic to a regular neighborhood of $K^+_T$.  If we repeat the process for the highest vertex in $K_1$ and so on, the result is a graph $K_m$ with all its vertices in $K^+_\Sigma$ and no other locally maximal sublevels.  
\begin{figure}[htb]
  \begin{center}
  \includegraphics[width=4.5in]{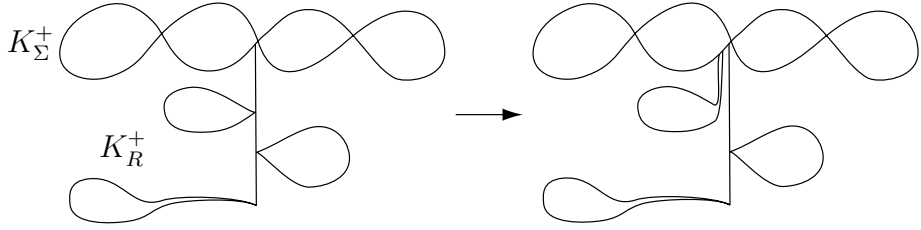}
  \put(-345,65){$K^+_\Sigma$}
  \put(-310,25){$K^+_R$}
  \caption{In the graph $K^+_T$ containing $K^+_R$ and $K^+_\Sigma$, we pull each vertex in $K^+_T$ up into $K^+_\Sigma$ by shrinking vertical edges.}
  \label{commonspinefig}
  \end{center}
\end{figure}

Each edge of $K_m$ outside of $K^+_\Sigma$ has the property that every point is connected to $v_0$ by a non-decreasing path.  Thus each edge must have a single minimum with respect to $f$.  The complement $H'$ of a regular neighborhood of $K^+_\Sigma$ is a handlebody isotopic to $H^-_\Sigma$ and each edge of $K_m$ intersects this handlebody in a boundary parallel arc (since it has a single minimum with respect to $f$).  Thus the complement in $M$ of a regular neighborhood of $K_m$ is a handlebody.  By construction, a regular neighborhood of $K_m$ is isotopic to a regular neighborhood of $K^+_T$, so $K^+_T$ defines a Heegaard splitting $(T, H^-_T, H^+_T)$.  As noted above this Heegaard splitting has genus $p + q + n - 1$ and is a common stabilization of $\Sigma$ and $R$.
\end{proof}

\section{Handlebodies and h-compressions}
\label{hcomsect}

Let $M$ be a closed 3-manifold and $H$ a handlebody or a collection of handlebodies embedded in $M$.  We will say that a surface $S$ is \textit{transverse} to $H$ if it is transverse to $\partial H$ and $S \cap H$ is a collection of disks in $S$ that are essential disks for $H$.  

Let $C$ be a component of a surface $S$ transverse to $H$.  If $C \setminus H$ is a sphere with at most two punctures then its \textit{complexity} is zero.  The complexity of a sphere with three punctures is one. Otherwise, the \textit{complexity} of $C$ is twice its genus minus its Euler characteristic minus one.  The complexity of $S$ will be the sum of the complexities of its components.  So for a surface $S$ with no sphere components, the complexity of $S$ will be twice the sum of the genera of its components minus the Euler characteristic, minus the number of components.  For each component, the Euler characteristic is non-positive and the genus is strictly positive, so the complexity is non-negative and is zero exactly when $S$ consists of spheres, each intersecting $H$ in at most two disks.

A \textit{trivial sphere} is a component of $S$ that consists of a sphere bounding a ball $B \subset M$ with interior disjoint from $S$ such that $H \cap B$ is either empty or a regular neighborhood of of an unknotted (i.e. boundary parallel) arc in $B$.  Two surfaces $S$, $S'$ will be called \textit{sphere-blind isotopic} if they are related by a sequence of isotopies transverse to $H$ and the following three types of moves, or their inverses:

\begin{enumerate}
  \item We may remove from $S$ a trivial sphere component.
  \item If $C$ is a sphere component (not necessarily trivial) disjoint $H$ then we can attach an embedded tube from $C$ to any other component of $S$.
  \item If $C$ is a sphere component that intersects $H$ in exactly two disks and $C'$ is a second component of $S$ such that a loop of $C' \cap \partial H$ is parallel in $\partial H$ to a loop of $C \cap \partial H$ then we may attach a tube between these two loops parallel to $\partial H$.
\end{enumerate}

Note that the second and third of these moves will be equivalent to the first move if $C$ is a trivial sphere.  However, if $C$ is not trivial then the resulting surface may not be isotopic to the original. All three moves produce a new surface with the same complexity as the original.

We will say that $S$ is \textit{strongly separating} if the components of $M \setminus S$ can be labeled $+$ and $-$ so that each component of $S$ is in the frontier of one component labeled $+$ and one component labeled $-$.  (If $S$ is connected then strongly separating is equivalent to separating.)  A choice of labels $+/-$ will be called a \textit{transverse orientation} and every strongly separating surface will have exactly two transverse orientations (since we have assumed $M$ is connected.)  Note that the three moves defining sphere-blind isotopy preserve the property of being strongly separating and there is a canonical way to project a transverse orientation from the original surface to the new surface.

Let $S$ be a closed, transversely oriented, strongly separating surface transverse to $H$.  A \textit{compressing disk} for $S$ (with respect to $H$) is a disk $D$ disjoint from $H$ whose boundary is an essential loop in $S \setminus H$ and whose interior is disjoint from $S$.  Compressing $S$ across $D$ produces a new strongly separating surface transverse to $H$, and the transverse orientation on $S$ defines a transverse orientation on the new surface.  Note that if we were to instead compress $S$ along a second disk $D'$ with the same boundary as $D$, the resulting surface would be sphere-blind isotopic (though not necessarily isotopic) to the result of compressing along $D$.

A \textit{bridge disk} for $S$ (with respect to $H$) is a disk $D$ whose boundary consists of an arc in $S$ and an arc in $\partial H$ such that each arc connects two distinct components of $S \cap H$.  Moreover, we require that the arc is not in a twice-punctured sphere component of $S$.  Isotoping $S$ across such a disk, then isotoping the surface further to remove any resulting trivial disk of intersection, produces a new surface that is still transverse to $H$ and is still transversely oriented.  This is called a \textit{bridge compression} of $S$, and as with compression the boundary of $D$ determines the new surface up to sphere-blind isotopy.

A \textit{cut disk} for $S$ is a disk that intersects $H$ in a single essential disk, whose boundary is an essential loop in $S \setminus H$ and whose interior is disjoint from $S$.  Compressing $S$ across such a disk produces a new strongly separating surface transverse to $H$ we will call this a \textit{cut compression} of $S$. Because of move 3, the boundary of a cut disk determines the resulting surface up to sphere-blind isotopy.

An \textit{h-disk} for $S$ is either a compressing disk a bridge disk, or a cut disk for $S$.  We will see below that each of the three types of compressions defined by these three types of disks reduces the complexity of the surface.  An \textit{h-compression} is any one of these three moves.

Following the setup in~\cite{axiomatic}, we will define $\mS(M, H)$ as the cell complex whose vertices are sphere-blind isotopy classes of closed, strongly separating surfaces transverse to $H$ and whose edges correspond to h-compressions.  In other words, we choose a representative $S$ for each sphere-blind isotopy class $v$.  For each isotopy class of loop bounding an h-disk in $S$, we choose an h-disk $D$ with this boundary and include in $\mS(M, H)$ an edge from $v$ to the isotopy class represented by the surface that results from h-compressing along $D$.  While there may be non-isotopic h-disks with the isotopic boundary, the results of h-compressing across these disks will be sphere-blind isotopic, and thus determine a unique second vertex in $\mS(M, H)$.

We would like to include a face whenever two h-disks are disjoint or when two compressing disks in a torus component disjoint from $H$ intersect in a single point.  For each edge $e$ in $\mS(M, H)$ below a vertex $v$, with second endpoint $v'$, we have representatives $S$, $S'$ for $v$, $v'$ and a representative $D$ for the h-compression that makes $S$ isotopic to $S'$.

If $e_2$ is a second edge below $v$ such that the boundary of the disk $D_2$ representing $e_2$ can be made disjoint from $D$.  Then after this compression, $\partial D_2$ is contained in $S''$.  If this loop or arc is essential in $S''$ then the isotopy from $S''$ to $S'$ takes $\partial D_2$ to either a trivial loop in $S'$ or the boundary of an h-disk for $S'$, representing an edge $e'_2$ below $v'$. In the first case, we will say that $e_2$ \textit{projects to a point}. In the second case, we will call the edge $e'_2$ the \textit{projection} of $e_2$ across $e$.

Note that in the construction above, there may be many inequivalent isotopies from the surface that results from the h-compression to the representative for that isotopy class. A 2-cell is defined by choosing one such representative for each edge, and we will include in $\mS(M, H)$ a 2-cell for each pair of choices.

If the image of $\partial D_2$ is trivial in $S'$ then we will let $e'_2$ be the vertex $v'$, which we will think of as a length-zero edge from the vertex to itself.  There is also a projection $e'$ of $e$ below the vertex at the other end of $e_2$.  The vertices at the lower endpoints of $e'$ and $e'_2$ result from compressing $S$ across the disks $D$, $D_2$ and are thus sphere-blind isotopic.  If both $e$ and $e_2$ project to edges then the four edges form a loop, which we will fill in with a 2-cell in $\mS(M, H)$. If one or both edges project to a point, then we get a triangle or a bigon loop, which we will still fill in with a 2-cell.

From the 2-dimensional cells, we can define higher dimensional cells by projecting as in~\cite{axiomatic}, as follows:

Assume $v$ is a vertex in $\mS(M, H)$ with edges $e_1, e_2, e_3$ below $v$ so that any pair of these edges are contained in a 2-cell below $v$.  For simplicity, we will initially assume that all these 2-cells are quadrilaterals.  Then $e_1$, $e_2$ correspond to disjoint h-disks $D_1$, $D_2$ for $S$ and the second endpoint of $e_3$ corresponds to a surface $S'$ that results from h-compressing $S$ along a third h-disk $D_3$ disjoint from $D_1$ and $D_2$.  The disks $D_1$ and $D_2$ have disjoint boundaries in $S'$. If both project to edges then their projections $e'_1$, $e'_2$ along $e_3$ define a 2-cell below $v'$. If one or both project to points, then the face defined by $e_1$, $e_2$ will project to either a single edge below $v'$ or a point.

Projecting each pair of edges across the third edge determines three 2-cells, each below the second vertex of one of the edges $e_1$, $e_2$, $e_3$.  The lowest vertex on each of these 2-cells is the result of h-compressing $S$ along each of the disks $D_1$, $D_2$, $D_3$ in different orders.  However, the resulting surface is the same regardless of the order of compressions.  Thus the three 2-cells defined by pairs of edges in $e_1$, $e_2$, $e_3$ and the three 2-cells that result from projecting form the boundary of a cube below $v$.  We will include in $\mS(M, H)$ a 3-cell bounded by this cube.

If one or more of the 2-cells below $v$ is not a quadrilateral, we can perform a similar construction, with a resulting collection of 2-cells isomorphic to the result of crushing some faces and edges of a cube.  Given four edges below a vertex $v$ such that any two of the edges determine a 2-cell, we have constructed a 3-cell containing each subset of three of them.  We can further use projection to find a collection of 3-cells forming the boundary of a (possibly crushed) 4-cube below $v$ and insert a 4-cell into $\mS(M, H)$ below $v$, then repeat the process for each successive dimension.  The cell complex $\mS(M, H)$ will be the union of all such cells.  

The \textit{descending link} $L_v$ of $v$ is the simplicial quotient of the subcomplex of the link spanned by the vertices corresponding to edges below $v$. By the \textit{simplicial quotient}, we mean the simplicial complex that results from identifying any two simplices in the link that have the same boundary. This is necessary because, for example, a pair of vertices in the link may be spanned by an infinite number of edges.

The descending link of a vertex is made up of the ``corners'' of the cells below that vertex.  Because each $n$-cell in a height complex $\mS$ is defined by $n$ edges below a given vertex of $\mS$, the descending link is a simplicial complex, i.e. each cell in the descending link is a convex hull of its vertices. (In fact, it's a flag complex.)

In the following sections, we will recall the axioms defined in~\cite{axiomatic} and check that they are satisfied by $\mS(M, H)$.

\section{The thin position axioms}
\label{thinaxiomsect}

The following six axioms refer to a cell complex $\mS$ with oriented edges and a complexity function $c$.  In~\cite{axiomatic}, we showed that with the proper interpretation, the standard results about thin position (such as those in~\cite{schtom}) can be deduced from these six axioms.  We will prove in the following sections that $\mS(M, H)$ satisfies these axioms, so that we can use the resulting Theorems from axiomatic thin position. \\

\noindent
\textbf{The Net Axiom}: For any vertex $v \in \mS$, there is an integer $\ell(v)$ such that every edge path starting at $v$, along which the complexity strictly decreases, has length at most $\ell(v)$. \\

As in~\cite{axiomatic}, the 2-cells defined for $\mS(M, H)$ are all quadrilaterals, triangles and bigons with a unique maximum and minimum (with respect to the complexity).  Projections between edges are determined entirely by the 2-cells, and these in turn determine all the higher dimensional cells.  This is the basis for the following axiom: \\

\noindent
\textbf{The Morse Axiom}: Every 2-cell in $\mS$ is a diamond, a triangle or a bigon with a single local maximum.  Given three edges such that any two bound a 2-cell, the projection of any two across the third will determine a face, an edge or will project to a point.  Every $n$-cell $C$ is defined by mapping the boundary of an $n$-cube into the $n-1$-skeleton via projections. \\

The interior of each h-disk $D$ representing an edge $e$ below $v$ is contained in either the positive or negative complement of the surface $S$.  We will orient each edge so that it faces towards $v$ if $D$ is on the negative side of $S$, and away from $v$ if $D$ is on the positive side of $S$. An edge and its projection will be \textit{parallel-oriented} if either they both point up (with respect to the complexity) or they both point down. A 2-cell is \textit{parallel-oriented} if all its edges and their projections are parallel-oriented. \\

\noindent
\textbf{The Parallel Orientation Axiom}: For any 2-cell $q$ in $\mS$, the orientations on the edges of $q$ make it a parallel-oriented diamond, triangle or bigon. \\

The next two axioms deal with paths in $\mS(M, H)$. A path is \textit{oriented} if the orientation on each edge points from a given vertex to the next vertex in the path. The edges in a \textit{reverse-oriented} path all point towards the previous vertex.

By the Parallel Orientations axiom, if an oriented path contains two edges in a face of $\mS(M, H)$ then we can form a new path by replacing the two edges by the other two edges in the boundary.  If the initial or the final two edges are adjacent to the top vertex in the of the face then this construction either creates or eliminates a local maximum in the path.  Such a move will be called a \textit{vertical slide}.  

A move in which the initial and the final path both pass through the top and bottom vertices will be called a \textit{horizontal slide}.  Such a move does not affect the local maxima and minima in the path.  Two paths will be called \textit{equivalent} if they are related by a sequence of horizontal slides. 

We will say that a vertex $v$ in $\mS(M, H)$ is \textit{compressible} if there is an edge below $v$, i.e. the descending link of $v$ is non-empty.  Moreover, we will say that $v$ is compressible to the \textit{positive (negative)} side if there is an edge below $v$ pointing away from (towards) $v$. For a given path $E$ with local maximum $v$, the \textit{path link} is the subcomplex of the descending link spanned by the edges that can appear before or after $v$ in a path equivalent to $E$. \\

\noindent
\textbf{The Casson-Gordon Axiom}: Let $v$ be a maximum in an oriented path $E$, and let $v_-$, $v_+$ be the minima of $E$ right before and after $v$, respectively.  If $v_-$ is compressible to the positive side then either the path link of $v$ is contractible or $v_+$ is compressible to the positive side.  Similarly, if $v_+$ is compressible to the negative side then either the path link of $v$ is contractible or $v_-$ is compressible to the negative side.  \\

\noindent
\textbf{The Barrier Axiom}:  Given any vertex $v \in \mS$, there are vertices $v_-$, $v_+$ and paths $E_-$, $E_+$ from, $v$ to $v_-$ and $v_+$, respectively such that the following hold: Any directed path descending from $v$ can be extended to a decreasing path ending in $v_+$ that is equivalent to $E_+$.  Any reverse-directed path descending from $v$ can be extended to a decreasing, reverse-directed path ending at $v_-$ that is equivalent to $E_-$. \\

\noindent
\textbf{The Translation Axiom}: Let $q^+$ be an $n$-cell in $\mS$ such that the  edges adjacent to the minimum vertex $v$ of $q$ all point away from $v$ and let $q^-$ be an $m$-cell in $\mS$ such that the  edges adjacent to the minimum vertex $v$ of $q$ all point towards $v$.  Then there is a unique $(n+m)$-cell $C$ isomorphic to $q^+ \times q^-$ such that $q^+ = q^+ \times \{v\}$ and $q^- = \{v\} \times q^+$, up to self-isotopies of the surface defined by the maximum vertex of $C$. \\

\section{The axioms for handlebody thin position}
\label{thinhandlesect}

Four of the six axioms follow relatively easily, as we will describe below.  The following two sections are devoted to showing that $\mS(M, H)$ satisfies the remaining two axioms.

\begin{Lem}
\label{netaxiomlem}
The complex $\mS$ satisfies the Net axiom.
\end{Lem}

\begin{proof}
To prove the Lemma, we must check that the complexity goes down under h-compression. The complexity of a surface is a non-negative integer so this will imply that any decreasing path is at most as long as the complexity of the original surface.

There are three types of h-compressions.  If we perform a (standard) compression on a surface $S$, we replace an annulus with two disks.  This increases the Euler characteristic of $S$ and either decreases the genus or increases the number of components, depending on whether the compression is non-separating or separating.  For a cut compression, an annulus is replaced by two punctured disks so the Euler characteristic stays the same.  However, as with (standard) compression, a cut compression either decreases the genus or increases the number of components, so the complexity will strictly drop.  For a bridge compression, the genus and number of components stay the same, but the Euler characteristic increases, so the complexity drops.
\end{proof}

\begin{Lem}
The complex $\mS$ satisfies the Morse axiom.
\end{Lem}

\begin{proof}
As noted in Lemma~\ref{netaxiomlem}, the complexity decreases along each edge, so each face has a single local maximum. By construction, the 2-cells are quadrilaterals, triangles or bigons and the higher dimensional cells are built from the 2-cells using the fact that given three edges that all cobound 2-cells, the projection of any one 2-cell along the third edge is a 2-cell, an edge or a point.
\end{proof}

\begin{Lem}
The complex $\mS$ satisfies the parallel projection axiom.
\end{Lem}

\begin{proof}
Let $c$ be a face in $\mS(M, H)$ below a vertex $v$.  The edges of $c$ adjacent to $v$ will point either towards $v$ or away from $v$ depending on which side of the surface the corresponding h-disks reside.  The projections of these edges are defined by the images of each disk after h-compressing across the other disk.  The image will reside on the same side of the new surface as the original disk, so the projected edge will point in the parallel direction.
\end{proof}

\begin{Lem}
The complex $\mS$ satisfies the translation axiom.
\end{Lem}

\begin{proof}
We will mimic the proof in~\cite{axiomatic}.  Let $v \in \mS(M, H)$ be a vertex that is the minimum of two cells $c_+$, $c_-$ on opposite sides of $v$.  Let $v_+$, $v_-$ be the maxima of these cells, with representatives $S_+$, $S_-$.  Because $v$ is the result of repeatedly compressing $S_+$, we can choose a representative $S$ for $v$ that intersects $S_+$ in subsurfaces whose complement in $S_+$ is a collection of disks and punctured disks.  We can choose a similar representative with respect to $S_-$.  Because both representatives are ambient isotopic, we can isotope $S_-$ so that it intersects $S$ in this way.  Moreover, we can isotope $S_-$ so that $S \setminus S_-$ is disjoint from $S \setminus S_+$.  Then the union of $(S_- \cup S_+) \setminus S$ and $S_+ \cap S_-$ is a surface that can be h-compressed down to $S$.  The vertex of $\mS(M,H)$ defined by $S$ is the maximum in a cell containing both $c_-$, $c_+$ so $\mS(M, H)$ satisfies the translation axiom. 

Uniqueness follows from an argument similar to that in~\cite{axiomatic}, and we will leave the details to the reader.
\end{proof}

\section{h-Compression bodies}
\label{boundaryaxiomsect}

In order to show that $\mS(M, H)$ satisfies the Barrier Axiom, we must consider more carefully the submanifolds of $M$ bounded by surfaces in $\mS(M, H)$.  

A \textit{K-graph} in a compression body $H$ is a properly embedded graph consisting of vertical arcs, boundary parallel arcs with their endpoints in $\partial_+ H$ and one-vertex trees that are parallel into $\partial_+ H$.  Following Taylor and Tomova~\cite{tomtay}, we will say that a connected graph is a \textit{vertical pod} if it becomes a K-graph after removing a single edge with an endpoint in $\partial_- H$ such that the union of this edge and any of the remaining edges must form a vertical arc.

We will call a graph $G$ in $H$ an \textit{h-graph} if there is a collection of cut disks in $H$, with respect to $G$ such that for some regular neighborhood $N$ of these cut disks, $G \setminus N$ is a collection of K-graphs and vertical pods for $H \setminus N$.

Given a disjoint union of handlebodies $H \subset M$, we will say that a compression body $C \subset M$ is an \textit{h-compression body} (with respect to $H$) if $\partial C$ is transverse to $H$ and $H \cap C$ is a regular neighborhood in $C$ of an h-graph.

\begin{Lem}
If $E$ is a directed path in $\mS(M, C)$ then there is an h-compression body $H$ whose positive and negative boundaries represent the endpoints of $E$.
\end{Lem}

\begin{proof}
Each edge in $\mS(M, H)$ corresponds to either a compression disk, a bridge disk or a cut disk $D$ of a surface $S$ producing a surface $S'$.  Between $S$ and $S'$, there is either a compression body or a surface cross an interval (which is technically also a compression body.)  For $D$ a compression disk, $H$ intersects this compression body in a regular neighborhood of vertical arcs.  For a bridge disk, this compression body intersects $S$ in the neighborhood of a collection of vertical edges and either an arc parallel into the positive boundary or the union of such an arc and a vertical arc, i.e. a pod handle.  

If $D$ is a cut disk then if we remove a neighborhood of $D$ from the compression body, the result will intersect $H$ in a neighborhood of a collection of vertical arcs.  Thus the compression body is again an h-compression body with respect to $H$.

If $C$ and $C'$ are h-compression bodies such that $\partial_+ C = \partial_- C$ then $C \cup C'$ is an h-compression body.  Thus if two paths determine h-compression bodies and their union is directed and decreasing then their union determines an h-compression body.  Above, we checked that the Lemma is true for a single edge, so it is true in general by induction. 
\end{proof}

The \textit{descending link} of a vertex $v \in \mS(M, H)$ is the portion of the link spanned by the edges below $v$.  We will say that $v$ has \textit{index zero} if its descending link is empty, i.e. there are no vertices below $v$.  The vertex $v$ will have \textit{index one} if its descending link is disconnected.  The set of h-disks on a given side of a surface $S$ representing $v$ form a connected (or empty) set.  Thus $v$ will have index one if and only if $S$ has h-disks on both sides and every h-disk on one side intersects every h-disk on the other side.  This condition is often called strongly irreducible (or weakly incompressible). 

We can also define higher index surfaces in terms of the homotopy groups of their descending links, but for this paper, we will not need this generalization.

Note that if $X$ is an h-compression body and $S$ is the result of h-compressing the positive boundary of $X$ some number of times then within $X$, $S$ will have h-disks on only one side, and thus cannot have index one.  We will use this fact below.

\begin{Lem}
\label{pathtohbodylem}
Two monotonic, directed paths in $\mS(M, C)$ determine blind isotopic H-compression bodies if and only if they are equivalent, i.e. related by horizontal slides.
\end{Lem}

\begin{proof}
It is straightforward to check that if two paths are related by a horizontal slide then they determine the same h-compression body.  Equivalent paths are related by a sequence of horizontal slides and thus define the same h-compression body by induction.

For the converse, note that given an h-compression body $X$ and a surface $S \subset X$ that results from h-compressing $\partial_+ X$ some number of times, then $S$ is h-compressible to exactly one side, so as noted above, $S$ is not an index-one surface.  This implies that no path determining the compression body $X$ can pass through an index-one vertex.  As noted in~\cite{axiomatic}, this implies that if two decreasing paths determine the same compression body then they are equivalent.  
\end{proof}

\begin{Lem}
The complex $\mS$ satisfies the barrier axiom.
\end{Lem}

\begin{proof}
By the net axiom, there is some finite length, descending directed path $E$ below any vertex $v$ such that $E$ ends at a vertex $v_+$ that is incompressible to the positive side. If $E'$ is a second directed, descending path from $v$, we can extend it to a finite path that also ends at a vertex that is incompressible to the positive side. 

Let $S$ be a surface representing $v$ and let $H$ be the h-compression body determined by $E$, with $S = \partial_+ H$.  Because the negative boundary of $H$ is h-incompressible, each of the h-disks for $S$ defined by the path $E'$ can be isotoped into $H$.  Thus the h-compression body $H'$ determined by $E'$ can be isotoped into $H$.  Because $\partial_- H'$ is h-incompressible, $\partial_- H'$ must be parallel to $\partial_- H$, so the two h-compression bodies are isotopic.  By Lemma~\ref{pathtohbodylem}, this implies that $E$ and $E'$ are equivalent paths.
\end{proof}

\section{The Casson-Gordon axiom}
\label{cassongordonsect}

The final axiom left to check is the Casson-Gordon axiom.  This axiom requires that certain path links in $\mS(M, H)$ be contractible.  Since every path link is a simplicial complex, this is equivalent to the statement that every homotopy group of the path link is trivial.  We first note the following property of h-compressible surfaces:

\begin{Lem}
\label{findcompressionlem}
If $S$ is an h-compressible surface in $(M, H)$ then $S$ admits either a compressing disk or a cut disk.
\end{Lem}

\begin{proof}
We need only check that if $S$ has a bridge disk then it has either a compressing disk or a cut disk.  Let $D$ be a bridge disk for $S$ and let $\alpha$ be the arc $D \cap S$.  By definition, the endpoints of $\alpha$ are in distinct loops $\ell_1, \ell_2$ in $S \cap \partial H$.  Moreover, these loops bound disk $D_1, D_2$ of $S \cap H$.  Let $N$ be a regular neighborhood in $S$ of the union $D_1 \cup D_2 \cup \alpha$.  Because $D_1$ and $D_2$ are disjoint, $N$ is a disk.  Let $D'$ be the result of pushing the interior of $N$ off of $S$ on the side containing $D$, then pushing the resulting disk across $D$ so that it intersects $H$ in a single disk, as in Figure~\ref{bridgetocutfig}.  
\begin{figure}[htb]
  \begin{center}
  \includegraphics[width=2in]{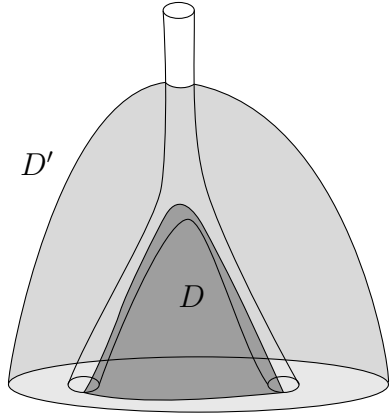}
  \put(-80,40){$D$}
  \put(-140,90){$D'$}
  \caption{Every bridge disk (shaded dark grey) is surrounded by a compression disk or a cut disk (shaded lighter grey).}
  \label{bridgetocutfig}
  \end{center}
\end{figure}

If the loop $D' \cap \partial H$ is essential in $\partial H$ then $D'$ is a cut disk for $S$.  Otherwise, we will isotope the disk further to remove the loop of intersection and create a compressing disk.  By assumption the arc $D \cap S$ is not in a twice-punctured sphere component of $S$, so $D' \cap S$ is essential and $D'$ is a cut disk or a compressing disk for $S$.
\end{proof}

This allows us to deal only with compressions and cut compressions in the following:

\begin{Lem}
\label{contrlinklem}
Assume $M'$ is the submanifold between consecutive thin levels $S_-$, $S_+$ of a path in $\mS(M, H)$ and $S$ is the thick surface between these thin levels.  If $S_-$ or $S_+$ is h-compressible into $M'$ then the path link of $S$ is $n$-connected for every integer $n$.
\end{Lem}

The proof of this Lemma is almost identical to the proof of Lemma~10.1 in~\cite{axiomatic}, so we will only present the outline here.  The reader can refer to the original paper for details.

Without loss of generality, assume that $S_-$ is compressible into $M'$ and let $D$ be an h-disk for $S_-$.  By Lemma~\ref{findcompressionlem}, assume $D$ is either a compressing disk or a cut disk.  Let $\Phi$ be an immersed $n$-sphere in the path link $L_v$ for $S$.  Each vertex corresponds to an h-disk $E$ for $S$ and defines a sequence of surfaces isotopic to $S$ that result from shrinking $E$ so as to remove components of $E \cap D$, one at a time.  Each $n$-cell in $\Phi$ defines a cube of surfaces defined in this way, and the union of these cubes forms a ball $B$.

There is a natural way to triangulate each of these cubes that defines a triangulation for $B$ and we can extend the triangulation so that its boundary can be naturally associated with the $n$-sphere $\Phi$.  Following~\cite{axiomatic}, we will call the resulting complex a \textit{Bachman ball}.  

Each vertex in the ball represents a surface in $M'$ isotopic to $S$.  Because the negative boundary of an h-compression body is h-incompressible, the surface $S$ cannot be made disjoint from $D$.  Thus the intersection of each surface with $D$ must contain one of more loops that are essential in $S$. An innermost (in $D$) such loop will be the boundary of an h-disk contained in $D$.  By choosing one of these disks for each vertex of $B$, we get a map from the vertices of $B$ to the vertices of the descending link of $S$.  Moreover, because of the way we triangulated $B$, this map extends to a continuous map from all of $B$ into the descending link, implying that the immersed sphere $\Phi$ is homotopy trivial.  Since $\Phi$ was arbitrary, this implies that the descending link is has trivial $n$th homotopy group.

\begin{Coro}
The complex $\mS$ satisfies the Casson-Gordon axiom.
\end{Coro}

\begin{proof}
Given $v_-$, $v_+$ and $v$ as in the axiom, the descending paths from $v$ determine the submanifold $M'$.  Without loss of generality, assume the positive descending link of $v_-$ is not empty.  Then by Lemma~\ref{findcompressionlem}, it has either a compressing disk or a cut disk on its positive side.  If this disk intersects the surface represented by $v_+$ non-trivially then $v_+$ is compressible to the positive side and the proof is complete.  Otherwise, we can make the disk disjoint from the other surface, so that $D$ is contained entirely in $M'$.  Then by Lemma~\ref{contrlinklem}, every homotopy group of the path link of $S$ is trivial, so the path link is contractible.
\end{proof}

\section{Interpreting $\mS(M, H)$}

We have been interested in the complex of surfaces $\mS(M, H)$ relative to a pair of handlebodies $H$.  However, there is a much more simple complex $\mS(M) = \mS(M, \emptyset)$ that does not consider handlebodies at all.  In this complex, vertices are simply (sphere blind) isotopy classes of surfaces in $M$, and edges correspond to compressions.

Because vertices of $\mS(M)$ correspond to transversely oriented surfaces, there are two vertices in $\mS(M)$ corresponding to the empty surface.  For one of these vertices, $v_-$, the manifold $M$ is labeled with a $+$ and for the other, $v_+$, it is labeled with a $-$. An oriented path from $v_-$ to $v_+$ with a single maximum $v$ defines a Heegaard splitting because the path from $v$ to $v_+$ determines one handlebody and the reverse path from $v$ to $v_-$ determines a second handlebody.  A connected surface representing $v$ is the Heegaard surface for this splitting.  Following~\cite{axiomatic}, we call any path with endpoints $v_-, v_+$ and a single maximum a \textit{Heegaard path}.

There is a canonical map $\mS(M, H) \rightarrow \mS(M)$ that takes each surface transverse to $H$ to itself in $M$.  This map crushes each edge corresponding to a bridge disk down to a single point, as well as all the edges corresponding to compressions and cut compressions that separate a planar surface from a non-planar component.  Moreover, this map extends to the higher dimensional cells of $\mS(M, H)$ and preserves the relative complexities of the endpoints of the remaining edges, so it is a \textit{height homomorphism}, as defined in~\cite{axiomatic}.

Any two vertices in $\mS(M, H)$ representing the same vertex in $\mS(M)$ are related by a sequence of bridge compressions (and their inverses) that correspond to an isotopy from one representative to the other.  This idea can be extended to the following Lemma, which is left as an exercise for the reader:

\begin{Lem}
\label{reprepath}
For any path in $\mS(M)$, there is a path in $\mS(M, H)$ that is mapped onto the path.  Moreover, if the path in $\mS(M)$ is oriented then the path in $\mS(M, H)$ can be chosen to be oriented.
\end{Lem}

In other words, we can lift any path in $\mS(M)$ to a path in $\mS(M, H)$. The lifted path $E$ in $\mS(M, H)$ will not necessarily have a single maximum.  However, the translation axiom allows us to slide (i.e. amalgamate) $E$ to a Heegaard path in $\mS(M, H)$.  Moreover, Lemma~17.3 in~\cite{axiomatic} states that (for any height complex in which the translation axiom holds) any equivalence class of paths amalgamates to a unique equivalence class of Heegaard paths.

Because the map from $\mS(M, H)$ to $\mS(M)$ preserves the relative complexities along edges, the image of every Heegaard path in $\mS(M, H)$ is a Heegaard path in $\mS(M)$ and thus determines a Heegaard splitting.  Moreover, the reader can check that equivalent paths in $\mS(M, H)$ will have equivalent images in $\mS(M)$, so we have the following:

\begin{Lem}
\label{uniquesplittinglem}
Every Heegaard splitting is represented by an oriented path in $\mS(M, H)$, every oriented path in $\mS(M, H)$ determines a unique isotopy class of Heegaard splittings in $M$ and any two equivalent paths in $\mS(M, H)$ determine the same isotopy class of Heegaard splittings.
\end{Lem}

\section{Flat surfaces}
\label{flatsect}

Let $\Sigma$ be a compact, connected, closed, orientable surface and let $N = \Sigma \times [0,1]$.  (The discussion below also be adapted to $\Sigma \times S^1$ and to surface bundles, but we will leave that for future work.)  A \textit{vertical annulus} in $N$ is a surface of the form $\ell \times [a,b]$ where $\ell$ is a simple closed curve in $\Sigma$ and $[a,b]$ is a closed interval in $[0,1]$.  A \textit{horizontal subsurface} in $\Sigma \times [0,1]$ is a surface of the form $F \times \{a\}$ where $F \subset \Sigma$ is compact subsurface and $a \in [0,1]$.  We will say that a compact, properly embedded surface $S \subset \Sigma \times [0,1]$ is \textit{flat} if $S$ is the union of a collection of vertical annuli and horizontal subsurfaces such that any two vertical annuli in $S$ are disjoint.  In particular, the vertical annuli have disjoint boundary loops which coincide with the boundary loops of the horizontal surfaces.

\begin{Lem}
Every pl surface properly embedded in $\Sigma \times [0,1]$ is isotopic to a flat surface.
\end{Lem}

\begin{proof}
Let $\pi$ be the projection map from $\Sigma \times [0,1]$ to $[0,1]$.  If $S$ is a piecewise-linear surface then (after isotoping $S$ slightly if necessary) the level sets of $f|_S$ will consist of simple closed curves and finitely many graphs in $S$.  We can isotope $S$ so as to make a regular neighborhood of each graph horizontal.  The complement will be foliated by simple closed curves, and thus consists of annuli, which can be isotoped to vertical annuli.  The result of this isotopy is a flat surface.
\end{proof}

If $S$ is strongly separating and transversely oriented in $N$ then every horizontal subsurface $F \times \{a\}$ has a positive component of $N \setminus S$ on one side and a negative component on the other.  We will say that this subsurfaces \textit{faces up} if $F \times \{a + \epsilon\}$ is in the positive component of $N \setminus S$ for small $\epsilon$ and \textit{faces down} if it is contained in the negative component. Every subsurface will face up or down and we can assume that no level set $\Sigma_t$ contains both up-facing and down-facing subsurfaces of $S$.

We will say that two horizontal subsurfaces of $S$ are \textit{adjacent} if there is a vertical annulus with one boundary loop in each of the horizontal subsurfaces.  We will say that a vertical annulus $A$ adjacent to a horizontal subsurface $F$ \textit{goes up} if $A$ is locally above $F$ (with respect to $f$).  Otherwise, we will say that $A$ \textit{goes down} from $F$.

\begin{Def}
We will say that a flat surface $S$ is \textit{tight} if no horizontal annulus has one adjacent vertical annulus going up and the other going down and if $F$ is a horizontal disk or a horizontal annulus with both annuli going the same direction then the closest adjacent subsurface faces the opposite way from $F$.
\end{Def}

If $F$ is a horizontal annulus with one adjacent vertical annulus above it and the other below, then we can shrink $F$ to a loop and isotope the adjacent annuli so that the two vertical annuli form a single vertical annulus.

If $F$ is a horizontal disk or an annulus with both vertical annuli facing the same way such that the closest adjacent horizontal subsurface $F'$ faces the same way as $F$, then the projections of $F$ and $F'$ will be disjoint. We can push $F$ into the same level as $F'$, shrinking the annulus between them down to a loop, and pushing any other portions of $S$ out of the way. The resulting surface is still flat, but has one fewer horizontal subsurfaces. Thus we have the following:

\begin{Lem}
Every flat surface $S$ is isotopic to a tight surface $S'$ such that every horizontal subsurface of $S$ is either a disk, an annulus or is sent into a horizontal subsurface of $S'$.
\end{Lem}

We will often want to rule out horizontal disk and annuli subsurfaces in tight surfaces.  This is not possible in general, but the following Lemma will allow us to rule them out in many cases.

\begin{Lem}
\label{tightdisknotesslem}
If $S$ is a tight surface in which some horizontal subsurface is either a disk or an annulus then then there is a horizontal loop that is trivial in $\Sigma$ and essential in $S$.
\end{Lem}

\begin{proof}
Assume for contradition $S$ contains a horizontal disk or annulus subsurface $F$, but every horizontal loop is either trivial in both $S$ and $\Sigma$ or is non-trivial in both.  If $F$ is a disk then there is a single vertical annulus $A$ adjacent to $F$, and we will let $F'$ be the other subsurface adjacent to $A$.  If $F$ and $F'$ faced the same way, they would have disjoint projections into $\Sigma$.  Because $S$ is tight, $F'$ must face the opposite way from $F$.  

This implies that the projection of $F'$ onto the level surface $\Sigma_t$ containing $F$ must be contained in $F$, as in Figure~\ref{disksubsurfacefig}.  Thus $S$ contains a horizontal loop that is trivial in both $S$ and $\Sigma$, but non-trivial in the horizontal subsurface $F$.  A similar argument shows that if $F$ is an annulus then there is a horizontal loop that is trivial in both surfaces, but not in a horizontal subsurface.
\begin{figure}[htb]
  \begin{center}
  \includegraphics[width=2.5in]{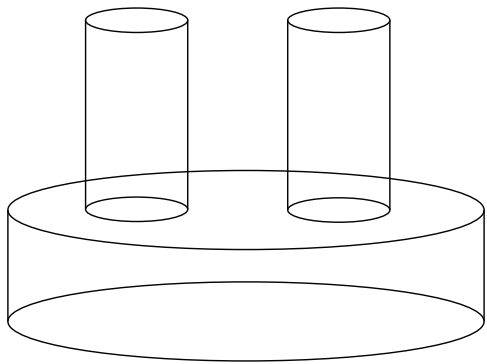}
  \caption{In a tight surface, if a horizontal subsurface is a disk then the adjacent subsurface has boundary loops that are trivial in $\Sigma$.}
  \label{disksubsurfacefig}
  \end{center}
\end{figure}

Consider the set of all horizontal loops that are trivial in both surfaces but not in an adjacent horizontal subsurface.  Each such loop bounds a disk containing at least one horizontal subsurface of $S$.  Let $D$ be an innermost such disk, with $F$ a horizontal disk subsurface in $D$ and $F'$ the adjacent horizontal subsurface as above.  Because $\partial D$ is non-trivial in a horizontal subsurface, $D$ must contain $F'$.  However, as noted above, $F'$ contains a loop that is trivial in both surfaces, contradicting the assumption that $D$ is innermost.  The contradiction completes the proof.  \end{proof}

Let $F_1, F_2 \subset S$ be horizontal subsurfaces facing opposite ways in consecutive levels $a, b$ of $[0,1]$ such that $a < b$.  Let $\alpha \subset \Sigma$ be an arc such that $\alpha \times \{a\}$ is properly embedded in $F_1$ with endpoints that end in vertical annuli that are both above $F_1$.  If $(\alpha \times \{b\}) \cap F_2$ is empty or consists of a regular neighborhood in $\alpha$ of one or both of its endpoints then the product $\alpha \times [a,b]$ is a disk, shown in Figure~\ref{bandmovefig}, whose boundary consists of an arc in $S$ and a horizontal arc $\alpha' \times \{b\}$ where $\alpha'$ is the closure of $\alpha \setminus F_2$.
\begin{figure}[htb]
  \begin{center}
  \includegraphics[width=5in]{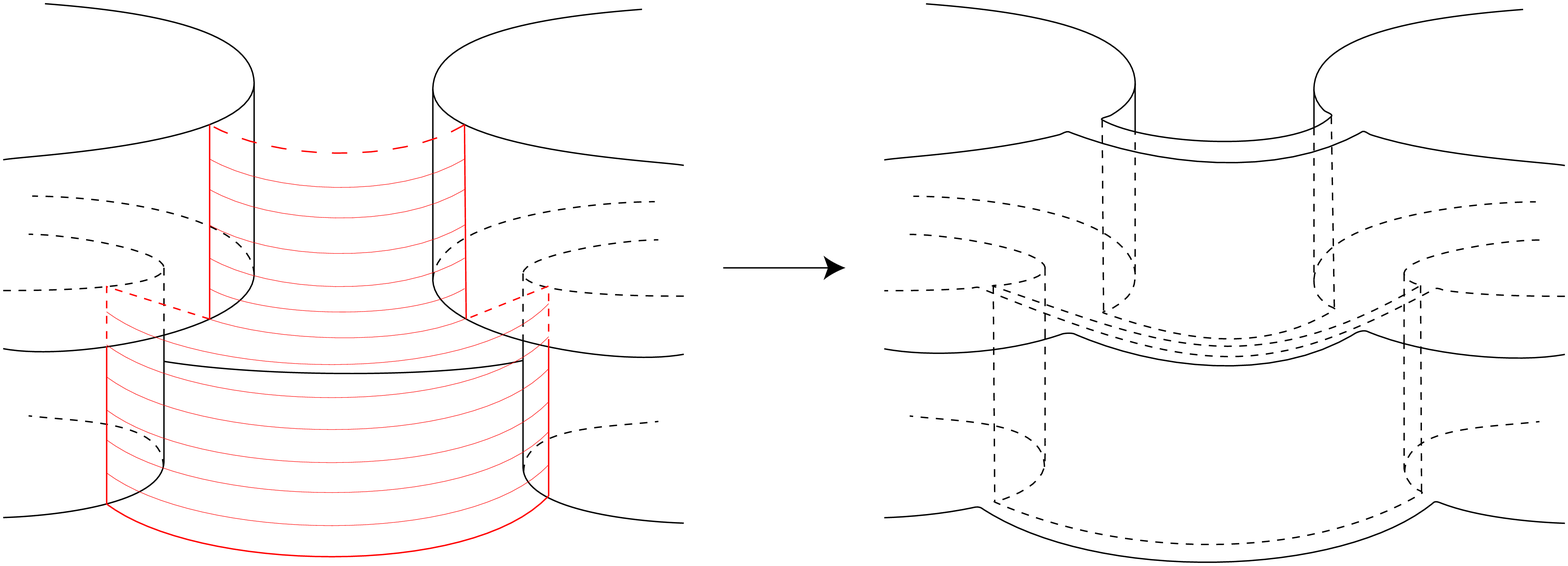}
  \put(-250,0){$\alpha$}
  \caption{The red disk defines a band move from the surface on the left to the surface on the right.}
  \label{bandmovefig}
  \end{center}
\end{figure}

The endpoints of $\alpha'$ are contained either in one or two vertical annuli above $F_2 \times \{b\}$, or in the interior of annuli above $F_1$ that extend past $F_2$.  If $c$ is the next level above $b$ containing a horizontal subsurface of $S$ then $\alpha' \times [b,c]$ is a disk whose boundary intersects $S$ in two vertical arcs.  The union $D = (\alpha \times [a,b]) \cup (\alpha' \times [b,c])$ is a disk whose boundary consists of an arc in $S$ and the horizontal arc $\alpha' \times \{c\}$.  Let $S'$ be the result of isotoping $S$ across the disk $D$, to form a new flat surface, then pulling this flat surface tight.  The surface $S'$ resulting from the arc $\alpha$ on the left of Figure~\ref{bandmovefig} is shown on the right of the Figure. 

\begin{Def}
We will say that $S'$ is the result of a \textit{band move} on $S$.
\end{Def}

The terminology comes from the idea that we should think of the move as transferring a band from one horizontal subsurface to the other. There is also, of course, a symmetric move in which we push a band down rather than up.

In the case above, the band from $F_1$ has to move through $F_2$ (rather than stopping at that level) because the band faces the opposite way from $F_2$. If $F_1$ and $F_2$ face the same way in consecutive levels $a, b$ of $f$, there is a simpler type of band move that moves a band from $F_1$ into $F_2$ as well as a move that takes a band from $F_2$ into $F_1$.

We will also define a type of band move under slightly weaker circumstances.  Assume there is an essential subsurface $F' \subset F_2$ such that the intersection $(\alpha \times \{b\}) \cap F'$ is either empty or consists of interval neighborhoods of one or both endpoints of $\alpha$.  Then we can form a new flat surface as follows:  Isotope the subsurface $F'$ of $F_2$ down to the level $b'$ half way between $a$ and $b$. The arc $\alpha$ now intersects the horizontal subsurface $F'$ above $F_1$ in a way that allows us to perform a band move, creating a non-empty up-facing level between $b'$ and $b$.  We will again say that this final surface $S'$ is a \textit{band move} of $S$.

In the rest of the paper, $N = \Sigma \times [0,1]$ will be the complement in a 3-manifold $M$ of a set $H$ consisting of two handlebodies.  A transverse surface $R$ in $M$ intersects $N$ in a properly embedded surface $S$, and we can isotope $R$ so that $S = R \cap N$ is flat.  In this picture, a compressing disk for $R$ (with respect to $H$) is a compressing disk in the usual sense for $S$.  A \textit{bridge disk} for $R$ appears as a boundary compressing disk for $S$ that intersects two different components of $\partial S$.  A cut disk appears as an annulus that has one component in $S$ and the other component in $\partial N$, i.e. in $\Sigma \times \{0\}$ or $\Sigma \times \{1\}$.  Such an annulus will define a cut disk for $R$ if and only if its boundary loop in $\partial N$ bounds a disk in one of the handlebody components of $H$.

\section{Essential surfaces}
\label{esssurfsect}

A flat surface $S$ will be called \textit{essential} if for every horizontal subsurface $F \subset \Sigma_t$ of $S$, the boundary of $F$ is essential in $\Sigma_t$.

\begin{Lem}
\label{esstightlem}
Assume $S$ is the intersection of $N$ with an index-zero (with respect to $H$) surface in $M$.  If $S$ is tight then $S$ is essential.  
\end{Lem}

\begin{proof}
If $S$ is not essential then by definition there is a horizontal subsurface $F \subset \Sigma_t$ whose boundary contains a trivial loop.  Because $S$ is tight, Lemma~\ref{tightdisknotesslem} implies that there a horizontal subsurface $F'$ whose boundary contains a loop that is trivial in $\Sigma$ but essential in $S$.  An innermost (in $\Sigma$) such loop bounds a disk in $N$, which defines a compressing disk for $S$, contradicting the assumption that $S$ is the intersection of $N$ with an index-zero surface.
\end{proof}

We will show in later sections that index-one surfaces are isotopic to surfaces that are either essential or have the follow form:

We will say that a subset $S \cap F_t$ of a surface $S \subset N$ has a \textit{flipped square} if it consists of the union of a subsurface $F$ of $S$ and a disk $D$ in $S$ such that $F \cap D$ consists of four points.  Moreover, the intersection of the vertical annuli just above $F_t$ with those just below $F_t$ is these same four points.  In such a subsurface, the normal vector to the disk points in the direction opposite the rest of the subsurface, as in Figure~\ref{flippedfig}.

\begin{figure}[htb]
  \begin{center}
  \includegraphics[width=3.5in]{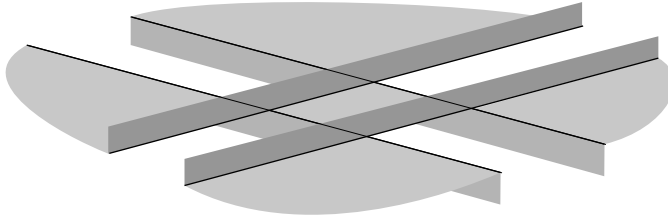}
  \caption{A flipped square is a horizontal disk that intersects the rest of the horizontal subsurface in four points and faces the opposite way.}
  \label{flippedfig}
  \end{center}
\end{figure}

\begin{Def}
A surface $S \subset N$ is an index-one essential surface if it has one horizontal subsurface with a flipped square and the rest of the surface is essential.
\end{Def}

Higher index essential surfaces can be defined by allowing more than one flipped square in the level surfaces.  However, we will only need to consider index-one surfaces for the present work.

Note that the boundary loops of the vertical annuli in an index-one essential surface are essential in their respective level surfaces.  However, the level subsets $S \cap F_t$ may contain loops that are trivial in $F_t$ but essential in $S \cap F_t$.  (If all these loops are essential then we can isotope $S$ to an essential surface without a flipped square.)

We will see below that index-one surfaces with respect to a pair of handlebodies can be isotoped to be index-one essential surfaces.  They may also be isotopic to essential surfaces without flipped squares, but not always.  This is analogous to the fact that to normalize a Heegaard surface in a triangulation, one needs to allow almost normal pieces. (See \cite{rub:alnormal},~\cite{stocking}).

\section{Disks and band moves}
\label{bandmovesect}

Consider an h-disk $D$ for a tight surface $S \subset N$.  The boundary of $D$ consists of horizontal arcs in the horizontal subsurfaces of $S$ and potentially more complicated arcs in the vertical annuli.  However, we can always isotope every essential arc in the vertical annuli to be transverse to the level surfaces $\Sigma_t$ and we can isotope any trivial arc out of the vertical annuli.  Similarly, we can isotope any trivial arc in a horizontal subsurface out of that subsurface.  Thus we will assume that $\partial D$ consists of a union of horizontal arcs that are essential in horizontal subsurfaces of $S$ and vertical arcs in the vertical annuli.  If $D$ is a bridge disk then one horizontal arc in its boundary will be contained in $\partial N$.

We can further isotope $D$ so that for the projection map $\pi : \Sigma \times [0,1] \rightarrow [0,1]$, the restriction of $\pi$ to the interior of $D$ is a Morse function whose level sets consist of loops, properly embedded arcs and saddles.  (A saddle is a graph with one valence four vertex in the interior of $D$ and zero, two or four valence one vertices in the boundary.)

A \textit{tetrapod} is a saddle with four arcs ending in $\partial D$.  Such a level $\tau$ cuts $D \cap N$ into four disks if $D$ is a compressing disk or bridge disk, or three disks and an annulus if $D$ is a cut disk.  (See~\cite{Schl} for a more detailed description and analysis of level sets of Morse functions on disks.)  If three of the disks are disjoint from $\partial N$ and do not contain any saddles then we will say that $\tau$ is an \textit{outermost tetrapod}.  A standard outermost disk argument, as in~\cite{Schl}, shows that if there is a tetrapod in $D$ then there is an outermost tetrapod.

\begin{figure}[htb]
  \begin{center}
  \includegraphics[width=2in]{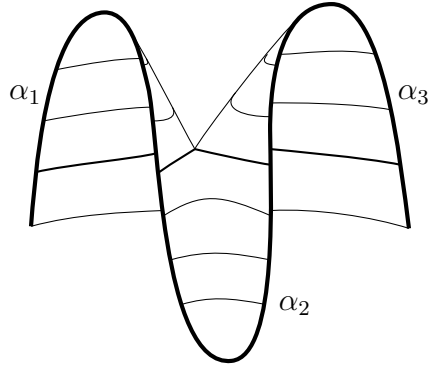}
  \put(-152,100){$\alpha_1$}
  \put(-50,20){$\alpha_2$}
  \put(-5,100){$\alpha_3$}
  \caption{An outermost tetrapod, shown as the second-thickest collection of lines, determines three disks, two above the tetrapod and one below, or vice versa.  The thickest line is the boundary of the disk.  For simplicity, the boundary of $D$ is shown as smooth rather than piecewise linear.}
  \label{tetrapodfig}
  \end{center}
\end{figure}

\begin{Lem}
\label{tetrapodisotopylem}
If a disk $D$ for a tight surface $S$ contains an outermost tetrapod $\tau$ then there is a sequence of band moves of $S$ after which we can isotope $D$ in $N$ to eliminate $\tau$.
\end{Lem}

\begin{proof}
Let $D_1$, $D_2$, $D_3$ be the disks in the complement of $\tau$ that are disjoint from $\partial N$ and from all the tetrapods other than $\tau$.  Assume $D_2$ is adjacent to $D_1$ and $D_3$.  Let $\alpha_i$ be the arc $\partial D_i \cap \partial  D$ for each $i$.

Without loss of generality, assume that the horizontal arc $\alpha_2$ is below $D_2$, as in Figure~\ref{tetrapodfig}.  Then by construction, the horizontal arcs $\alpha_1$, $\alpha_3$ are above $D_1$ and $D_3$.  The tetrapod $\tau$ is contained in a level surface $\Sigma \times \{a\}$ sitting between two horizontal levels of $S$.  Each disk $D_i$ intersects every level between $\alpha_i$ and $\tau$ in an arc parallel to the projection of $\alpha_i$.  If one of these levels contains a horizontal subsurface of $S$ then the projection of $\alpha_i$ will be disjoint from it or intersect it in a neighborhood of one or both of its endpoints.  Thus there is a band move of $S$ that moves each $\alpha_i$ past each level between it and $\tau$.  

In the surface $S'$ that results from these band moves, the arcs $\alpha_1$, $\alpha_3$ sit in one level surface of $S'$ and $\alpha_2$ sits in the level subsurface surface just below them.  The projections of $\alpha_1$, $\alpha_2$, $\alpha_3$ into the level surface $\Sigma \times \{a\}$ containing the tetrapod $\tau$ are parallel to subgraphs of $\tau$, so their projections are isotopic to pairwise disjoint arcs.  The vertical disk defined by $\alpha_2$ will thus intersect a regular neighborhood of $\alpha_1 \cup \alpha_3$ in a regular neighborhood of the the endpoints of (the projection of) $\alpha_2$.  

Split the horizontal surface containing $\alpha_1 \cup \alpha_3$ by pushing a regular neighborhood of $\alpha_1 \cup \alpha_3$ down towards the horizontal subsurface containing $\alpha_2$.  Because the vertical disk defined by $\alpha_2$ intersects this subsurface in a neighborhood of its endpoints, there is a band move taking a band along $\alpha_2$ past the horizontal surface containing $\alpha_1 \cup \alpha_3$.  After this move, $\alpha_2$ is above $\alpha_1$ and $\alpha_3$ so we can modify $D$ to remove the tetrapod.  Note that this entire construction does not create any new tetrapods in $D$.
\end{proof}

\section{Index-one surfaces}
\label{indexonesect}

If a tight surface $S$ contains a horizontal subsurface with a boundary loop that is trivial in $\Sigma$ then an innermost (in $\Sigma$) such loop or arc bounds a compressing disk or bridge disk (respectively) for $S$.  We will call these \textit{horizontal disks}.  Distinct horizontal disks have disjoint boundaries and any two disjoint compressing disks for an index-one surface must be on the same side of the surface.  Thus we have the following:

\begin{Lem}
\label{disksideslem}
If $S$ is a tight index-one surface then any two horizontal disks for $S$ are on the same side of $S$.
\end{Lem}

We will use this in the proof below.

\begin{Lem}
\label{normalizerlem}
Assume $S$ is the intersection of $N$ with an index-one surface in $M$.  Then $S$ is isotopic to essential or index-one essential surface.
\end{Lem}

\begin{proof}
Let $S$ be a tight flat surface representing an index-one vertex $v$ in $\mS(M, H)$.  If $S$ is essential then we have found our essential representative of $S$.  Otherwise, let $D^-$, $D^+$ be a pair of h-disks in different components of the disk complex for $S$.  In particular, we will take $D^-$ on the negative side of $S$ and $D^+$ on the positive side.  By Lemma~\ref{findcompressionlem}, we can assume $D^-$ and $D^+$ are cut disks or compression disks.

If $D^+$ is not vertical or horizontal then it contains an outermost tetrapod and there is a sequence of band moves, given by Lemma~\ref{tetrapodisotopylem} defining a sequence of tight surfaces $S_0,S_{1},\dots,S_{k'}$ after which we can isotope $D^+$ to remove the tetrapod.  By repeating this process for each tetrapod in $D^+$, we can extend this sequence by further band moves until $D^+$ contains no tetrapods, and is thus vertical or horizontal.  If $D^+$ is a compression disk  and the final image of $D^+$ is vertical then there is a final band move that makes it horizontal.  Let $S_0,\dots,S_k$ be the resulting sequence of tight surfaces.  

Define a similar sequence $S_0,S_{-1},\dots,S_{-\ell}$ using the disk $D^-$.  We will show that either one of the surfaces $S_{-\ell},\dots,S_k$ is essential, or there is an essential surface with one flipped square that is intermediate between two of them.

The band move from $S_0$ to $S_1$ consists of three parts, the first of which is optional:  First, we are allowed to separate a horizontal subsurface of $S_0$ into two surfaces $F_0, F_1$, along a collection of essential curves, producing $S'_0$.  Next, we move a band from a second horizontal subsurface $F_2$ past the horizontal surface $F_1$, to get $S'_1$.  Finally, we make $S'_1$ tight to produce $S_1$.  Splitting a horizontal subsurface of $S_0$ does not eliminate any boundary loops of horizontal subsurfaces, so any horizontal disk for $S_0$ is isotopic to a horizontal disk for $S'_0$.  Similarly, pulling $S'_1$ tight does not produce new horizontal loops, so any horizontal disk for $S_1$ is isotopic to a horizontal disk for $S'_1$.

Let $E$ be the disk defining the band move from $S'_0$ to $S'_1$.  The move affects three horizontal levels, which we will label $F_- < F_0 < F_+$.  Without loss of generality, assume that the band is moved from level $F_-$ to $F_+$, so $E$ intersects $F_-$ in an essential horizontal arc, as in Figure~\ref{closeupfig}.  Let $A_+$ be the annulus or pair of annuli between $F_0$ and $F_+$ that intersect $E$.  Let $A_-$ be the annulus or annuli between $F_0$ to $F_-$ intersecting $E$.  The disk $E$ consists of two vertical bands, one with vertical boundary in $A_-$ and the other with vertical boundary in $A_+$.  The surface $S'_1$ contains annuli $A'_-, A'_+$ that result from pinching $A_-, A_+$ along these two bands, i.e. removing a neighborhood of $E$ from the annuli and then gluing in bands parallel to $E$.
\begin{figure}[htb]
  \begin{center}
  \includegraphics[width=3.5in]{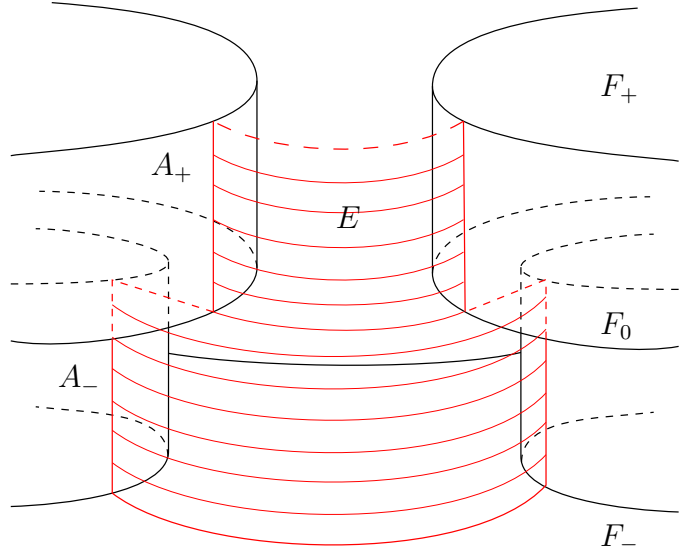}
  \put(-30,0){$F_-$}
  \put(-30,80){$F_0$}
  \put(-30,170){$F_+$}
  \put(-235,60){$A_-$}
  \put(-200,140){$A_+$}
  \put(-130,120){$E$}
  \caption{The labeling of the subsurfaces in the band move.}
  \label{closeupfig}
  \end{center}
\end{figure}

Because $S_0$ is not essential, there is a horizontal subsurface in $S_0$ that contains the boundary of a compressing disk.  If a horizontal subsurface disjoint from $A_-, A_+$ contains the boundary of a compressing disk $D_1$ for $S'_0$ then it will also contain a disk in $S'_1$, and thus for $S_1$.  This horizontal disk is disjoint from $D_0$ so there is an edge in the disk complex for $S$ from $D_0$ to $D_1$.  Otherwise, assume that every horizontal compressing disk for $S_0$ has boundary in $\partial A_-$ or $\partial A_+$.

Note that the horizontal subsurfaces of $S'$ are adjacent to $A_-$ on the same side as $E$, but adjacent to $A_+$ on the side opposite $E$.  Thus any horizontal compressing disk for $S'_0$ on the same side as $E$ is in $\partial A_+$, while any horizontal disk on the side opposite $E$ is in $A_-$.  Moreover, both of these cannot be the case since $S$ is an index-one surface, so any disjoint compressing disks are on the same side, by Lemma~\ref{disksideslem}.

In the case when $\partial A_+$ contains $\partial D_0$ on the same side as $E$, the band move cuts $D_0$ into two horizontal compressing disks, each of which can be isotoped disjoint from $D_0$ in $S$ (though not necessarily as horizontal disks).  Thus if we let $D_1$ be one of these disks, there will be an edge between them in the disk complex for $S$.

Otherwise, assume that $D_0$ is on the side opposite $E$ and $\partial D_0$ is contained in $\partial A_-$.  In particular, every other vertical annulus will have essential boundary in $\Sigma$.  If $A_-$ consists of two components, each bounding a horizontal disk, then the band move will turn the two horizontal disks into a single horizontal disk $D_1$ whose boundary can be isotoped away from $\partial D_0$.  Otherwise, the vertical annulus or one of the annuli $A'_-$ have essential boundary in $\Sigma$.

If we perform only the first half of the band move, which brings the band into the level of $F_0$, then the resulting surface contains a flipped square in this level, as in Figure~\ref{flippedbandfig}.  The vertical annuli in the surface consist of the vertical annuli in $S_0$, except that we have replaced $A_-$ with an essential vertical annulus or annuli $A'_0$.  The result is an essential surface with one flipped square.
\begin{figure}[htb]
  \begin{center}
  \includegraphics[width=5in]{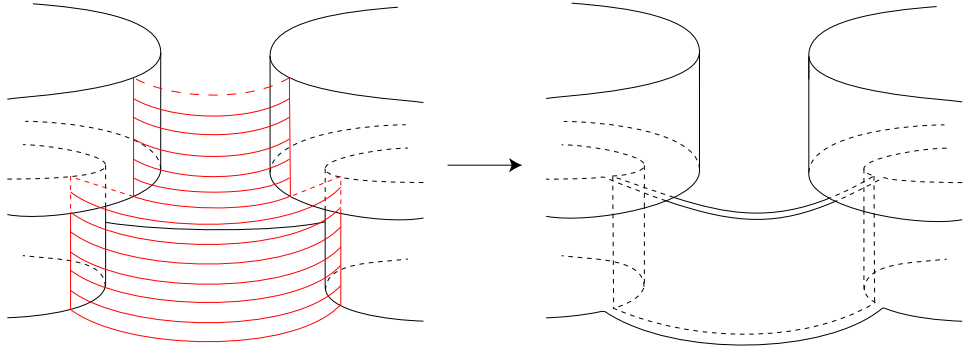}
  \caption{Stopping half way through the band move defined by the red disk defines a surface with a flipped square.}
  \label{flippedbandfig}
  \end{center}
\end{figure}

Otherwise, if the band move produces a surface $S_1$ with a horizontal disk $D_1$, we will repeat the argument for each $S_i$, then for each $S_{-i}$.  If each step produces a horizontal disk disjoint from the previous one, then we will have a path of disks from $D_k$ to $D_{-\ell}$.  In $S_k$ and $S_{-\ell}$, the disks $D^-$ and $D^+$ are either horizontal (in the case of compressing disks) or vertical (in the case of cut disks), and thus disjoint from $D_k$ and $D_{-\ell}$, respectively.  The resulting path from $D^-$ to $D^+$ contradicts the assumption that $D^-$ and $D^+$ are in distinct components of the disk complex for $S$.  Thus some $S_i$ must be an essential flat surface, possibly with a flipped square.  
\end{proof}

\section{Essential isotopies}
\label{essisosect}

Given a pair $H$ of handlebodies in $M$ whose complement is parameterized as $\Sigma \times [0,1]$, the thin position arguments in~\cite{axiomatic} show  that every Heegaard splitting for $M$ determines a thin path in $\mS(M, H)$, such that the local minima have index zero and the local maxima have index one. By Lemmas~\ref{esstightlem} and~\ref{normalizerlem}, the maxima and minima of these paths can be represented by essential surfaces.  We would like to fill in a family of essential surfaces related by a certain simple moves, which we will use in the next section to construct a nice spine for this Heegaard splitting with respect to $H$.

The first type of move is the band move already introduced above. For the second type of move, consider a flat surface $S$ containing a vertical annulus $A \subset S$ such that the boundary loops of $A$ bound horizontal disks $E_1$, $E_2$ with interiors disjoint from $S$.  If we replace $A$ with the two disks $E_1$, $E_2$, the result is a new flat surface $S'$.  We will say that $S'$ is the result of a \textit{horizontal compression} of $S$.  The inverse of this move consists of attaching a tube to $S'$ to get $S$.  In the statement of the Lemma below, we do not consider the ``direction'' of the move, so two surfaces are related by horizontal compression if and only if they are related by tubing.

For the third move, we will say that a component $C$ of a horizontal surface $S$ is a \textit{simple surface} if $C$ is the union of two horizontal subsurfaces and a collection of vertical annuli such that each vertical annulus goes from one horizontal subsurface to the other.

A \textit{parallel surface} is a component $C$ of $S$ consisting of a single horizontal subsurface and some number of vertical annuli with boundary loops in the same component of $\partial N$.  If $S'$ is the complement in $S$ of a simple surface or a parallel surface then we will say that $S'$ is the result of \textit{removing} a simple surface or parallel surface, respectively.  The reverse of this move is adding a simple surface or parallel surface.

\begin{Lem}
\label{esspathlem}
Let $E$ be a thin path in $\mS(M, H)$.  Then there is a sequence of tight surfaces $\{S_i\}$ in $N \setminus H$ representing a path equivalent to $E$ such that consecutive surfaces are related by band moves, horizontal compressions and removing simple surfaces and parallel surfaces. Moreover, each $S_i$ will be either essential or related to one of $S_{i-1}$, $S_{i+1}$ by a band move and to the other by a horizontal compression.  
\end{Lem}

\begin{proof}
As noted above, we can choose a sequence of essential or index-one essential surfaces representing the maxima and minima of the path $E$.  Because these surfaces correspond to a directed path in the complex of surfaces, we can choose them so that they are pairwise disjoint.  Our goal will be to fill in the family of surfaces $\{S_i\}$ between these.

Consider the initial vertex $v_0$ of $E$ and the first maximum $v_1$.  Let $S_0$ be a representative of $v_0$ and let $S$ be a surface representing $v_1$.  If $S$ does not contain a flipped square then we will skip to the next step.  If there is a flipped square, we can eliminate it by pushing the square to the negative side of $S$, treating it like a band move.  In general, the resulting surface $S'$ will have a vertical annulus whose boundary is trivial in $\Sigma$, and we will add this surface to the squence, followed by the essential surface that results from the horizontal compression defined by this annulus.

Let $D_1,\dots,D_\ell$ be a complete system of h-disk on the negative side of $S$.  First assume that some $D_j$ is not vertical, i.e. contains a level tetrapod.  By Lemma~\ref{tetrapodisotopylem}, there is a band move of $S$ defined by this tetrapod.  If the surface created by this band move is not essential then it has a horizontal level loop bounding a disk on the negative side and we will preform a horizontal compression across this disk, then choose a new system of disks $\{D_i\}$.

If no horizontal compressions appear, we will continue removing tetrapods from the disks $\{D_j\}$. Because the disks are disjoint, a band move defined by one disk will not affect the others. Because there are finitely many tetrapods in each $D_j$, this process will either find a horizontal compression or terminate in a finite number of steps. Each time it finds a horizontal compression, then we will perform the horizontal compression, then choose a new collection of disks and run the process again.

Because a compression reduces the complexity of the surface, the process must eventually terminate without finding a compression. When this happens, every $D_j$ is a vertical disk, so $S$ must consist of h-incompressible pieces, simple surfaces and parallel surfaces.  Removing the simple surfaces and parallel pieces corresponds to h-compressing along all the vertical disks $\{D_j\}$.  By the barrier axiom, every path of h-compressions on the negative side must end at $v_0$ and must be equivalent to the initial segment of $E$. Thus we can get $S_0$ by removing the simple surfaces and parallel surfaces, then performing band moves to turn the h-incompressible components into $S_0$.  Since these components are incompressible, every surface in this sequence is essential. The reverse sequence of surfaces $S_0,\dots,S_k$ gives us a sequence from $S_0$ to $S_k = S$ satisfying the conclusion of the Lemma.

We can further repeat the process for the monotonic segment from $v_1$ to the following minimum $v_2$ and so on. In the case when a maximum is represented by an essential surface with no flipped squares, this essential surface corresponds to some $S_i$.  In the case when the maximum has a flipped square, there is no $S_i$ of this form.  Instead, the two surfaces that correspond to resolving the flipped square in different directions are consecutive surface $S_i$, $S_{i+1}$ with a band move between them.  In this case, there will be a vertical tubing (the opposite of a horizontal compression) that creates $S_i$ and a horizontal compression after $S_{i+1}$.

At the local minimum of the path $E$, there may be different representatives for the h-incompressible surfaces.  However, the surfaces will be isotopic, and any isotopy can be carried out by a sequence of band moves between tight surfaces.  Moreover, because these surfaces are h-incompressible, every tight representative will be essential.  Thus we can fill in the gap between representatives of the local minima by sequences of tight surfaces related by band moves so that the surfaces are pairwise disjoint.
\end{proof}

\section{The main theorem}
\label{mainthmsect}

\begin{Lem}
\label{thinthenesslem}
If $(\Sigma, H^-_\Sigma, H^+_\Sigma)$, $(R, H^-_R, H^+_R)$ are unstabilized Heegaard splittings for $M$ then there is a spine for $K^+_R$ with at most $\frac{1}{2}p + q$ locally maximal horizontal components with respect to the sweep-out $f$, where $p$ is the genus of $R$ and $q$ is the genus of $\Sigma$.
\end{Lem}

\begin{proof}
Let $H$ be the union of disjoint regular neighborhoods of spines for $H^-_\Sigma$ and $H^+_\Sigma$.  Because these are spines for a Heegaard surface, the closure of the complement of $H$ is homeomorphic to $\Sigma \times [0,1]$.

By Lemma~\ref{reprepath}, there is a path in $\mS(M, H)$ that represents the Heegaard splitting $(R, H^-_R, H^+_R)$.  Let $E$ be the thin path that results from weakly reducing the initial path to a strongly irreducible path.  The net axiom guarantees such a path, Lemma~8.4 in~\cite{axiomatic} guarantees that such a path will have index-one maxima and the Casson-Gordon axiom guarantees that it will have index-zero minima.  

By Lemma~\ref{esspathlem}, we can construct a sequence $(R_i)$ of surfaces that represent a path equivalent to $E$ such that consecutive surfaces are related by band moves, vertical and horizontal compressions and tubings and such that every surface is either essential or is between a band move and a compression/tubing.  By Lemma~\ref{uniquesplittinglem}, this path also represents the Heegaard splitting $(R, H^-_R, H^+_R)$.  We will use this sequence of surfaces to build an essential spine for $H^+_R$.

Each surface $R_i$ is separating in $M$ and the components of the complement in $M$ can be divided into two types: The \textit{positive side} $C_+(R_i)$ is the union of the components that contain $R_{i+1}$ and the \textit{negative side} $C_-(R_i)$ is the union of the components that contain $R_{i-1}$.  We will say that a properly embedded graph $K_i \subset C_-(R_i)$ is a \textit{spine} for $C_+(R_i)$ if the complement $C_-(R_i) \setminus K_i$ is a union of handlebodies.

We will build a spine $K_i$ for each $R_i$ by induction. The first non-empty surface $R_1$ is either flat or parallel, so $C_-(R_1)$ is a ball or a handlebody by construction. We let $K_1$ be the empty set, so that $C_-(R) \setminus K_1$ is a ball or a handlebody.

For the inductive step, assume that $K_{i-1}$ is an essential spine for $R_{i-1}$.  We have five cases to consider, based on the four moves that can produce $R_i$ from $R_{i-1}$: vertical tubing, band moves, horizontal compression, and adding or removing simple or parallel surfaces.

If the surfaces are related by a vertical tubing, we can always isotope $K_{i-1}$ transverse to the level surfaces $\Sigma_t$ to be disjoint from the ends of the tube.  Then $K_{i-1}$ will be properly embedded in the complement $C_-(R_i)$.  Moreover, the tubing adds a one-handle to the complement, so $C_-(R_i) \setminus K_{i-1}$ is also a collection of handlebodies.  Thus $K_i = K_{i+1}$ will be a spine for $R_i$.

Next, consider the case when the surfaces are related by a band move.  Such a move consists of three parts:  First, we have the option of cutting a horizontal level into two pieces.  Next we move a band between to levels facing the same way, past a level facing the opposite way.  Finally, we pull the resulting surface tight.  

We can push the endpoints of $K_{i-1}$ away from the loops along which we want to split the initial subsurface and disjoint from the band disk. The band move pushes the band into $C_+(R_{i-1})$ and away from $K_{i-1}$ so we can extend $K_{i-1}$ to keep it properly embedded after this move.  Moreover, if the band is separating and cuts off a planar surface, we will always slide the endpoints of the vertical edges into the non-planar component.  When we pull the surface tight, we can extend the vertical edges to have endpoints in the resulting surface.

Next, in the case of a horizontal compression, we will add a vertical edge dual to the compression.  Note that because $R_i$ is tight, the horizontal compression cannot cut off a sphere component disjoint from $H$. However, if the horizontal compression cuts off a parallel sphere component with no endpoints of $K_i$, then we will immediately remove this parallel sphere and will not add an edge to $K_i$.

In the case when we add a parallel or simple component to $S$, the complement $C_-(R_i)$ changes by the addition of a ball or handlebody. Thus we will leave $K_i$ equal to $K_{i-1}$ as we did in the initial step.

Finally, consider the case when we remove a simple surface or a parallel surface.  A simple surface bounds a handlebody whose spine sits in some $\Sigma_t$.  When we remove such a surface, we will construct $K_{i+1}$ by adding this spine to $K_i$, then extending any vertical edges with endpoints in the two horizontal subsurfaces so that their endpoints are in the spine.

A parallel surface component bounds a ball or a handlebody whose boundary is parallel into $H^-_\Sigma$ or $H^+_\Sigma$.  Let $\ell$ be a horizontal spine for this handlebody (a single vertex in the case of a sphere).  If the parallel surface is a sphere that does not contain any endpoints of $K_i$, then it has just resulted from a horizontal compression and we will remove the sphere without modifying $K_i$, as noted above.  Otherwise, we will let $K_{i+1}$ be the union of a spine $\ell \subset \partial N$ and $K_i$ after extending the vertical edges as in the case of a simple surface.

By construction, the final surface $R_k$ is the empty surface, so the spine $K^+_R = K_k$ for its complement is a spine for a Heegaard splitting. Moreover, this Heegaard splitting is an amalgamation of the generalized Heegaard splitting defined by the path $E$.  Because amalgamations are unique (See, for example, Lemma~17.4 in~\cite{axiomatic}.) $K^+_R$ is a spine for $(R, H^-_R, H^+_R)$.  

To bound the number of locally maximal horizontal components of $K^-_R$, we examine how the graph was constructed.  Adding a vertical edge to the graph does not create a local maximum. Replacing a simple surface or a parallel surface with a horizontal graph will create a local maximum in the case when all the vertical edges with endpoints in the simple surface or parallel surface are below the surface.  

The number of simple surfaces and non-spherical parallel surfaces is at most the genus of $R$.  Moreover, by switching the direction of the sweep-out $f$, we can assume that at most half of them are local maxima.  This corresponds to flipping $\Sigma$, i.e. replacing $(\Sigma, H^-_\Sigma, H^+_\Sigma)$ with $(\Sigma, H^+_\Sigma, H^-_\Sigma)$ (Note the change of signs.)  Or, equivalently, we can flip $R$, then change the roles of $H^-_\Sigma$ and $H^+_\Sigma$.  The stable genus of $(R, H^-_R, H^+_R)$ and $(R, H^+_R, H^-_R)$ is at most $2q$, so this stays below the desired bound.  Thus we can assume that there are at most $\frac{1}{2} p$ local maxima coming from simple surfaces and non-spherical parallel surfaces.

The only other way we can create a local maximum is by removing a parallel sphere component that contains endpoints of $K_i$, whose vertical annuli are above its horizontal subsurface. The horizontal planar surface is created by a band move.  By the instructions for handling band moves, if the band had separated the planar surface from a non-planar surface then we would have kept the endpoints on the non-planar side.  Thus the planar surface must have been created by a band move along a non-separating band in a multi-punctured torus.

Let $F_1,\dots,F_n$ be the set of parallel spheres that contain endpoints of edges, ordered in accordance with the heights (relative to $\Sigma$) of their level surfaces so that $F_n$ is the highest.  For each $F_j$, let $G_j$ be the multi-punctured torus that is cut to produce $F_j$.  Each $G_j$ will be a subsurface of a different $R_i$ than its corresponding $F_j$, but the projection of $G_j$ will contain the projection of $F_j$.  Moreover, because each $F_j$ contains vertical annuli above its subsurface, the projections of $F_j$ will be disjoint from each $G_k$ with $k < j$.

Let $X_k \subset \Sigma$ be the union of the projections of $G_1, \dots, G_k$.  Because the projection of $G_j$ is a multi-punctured torus and the projection of $F_j$ is a planar subsurface that results from cutting $G_j$ along a band, the projection of $F_j$ contains a loop that is non-separating in $F_j$, and thus in $X_j$.  Moreover, $F_j$ is disjoint from $X_{j-1}$ by the above argument, so the loop in $F_j$ is disjoint from the loops in each $F_k$ for $k < j$.  By induction, the union of all such loops is non-separating.  Since $X_n$ contains $n$ mutually non-separating loops, $X_n$ has genus at least $n$.  Since $X_n$ is a subsurface of $\Sigma$ we must have that $n \leq q$, i.e. there are at most $q$ maxima that come from simple spheres.

Thus in the spine constructed above, there are at most $\frac{1}{2}p + q$ locally maximal horizontal components.
\end{proof}

We can now prove that any two Heegaard splittings $(\Sigma, H^-_\Sigma, H^+_\Sigma)$ and $(R, H^-_R, H^+_R)$, of the same 3-manifold $M$ have a common stabilization of genus at most $\frac{2}{3}p + 2q - 1$, where $p$ and $q$ are the genera of $\Sigma$ and $R$.

\begin{proof}[Proof of Theorem~\ref{mainthm}]
By Lemma~\ref{thinthenesslem}, there is a spine $K^+_R$ for $H^+_R$ with at most $\frac{1}{2}p + q$ locally maximal horizontal components.  By Lemma~\ref{esspinelem}, this implies that the two Heegaard splittings have a common stabilization of genus $p + q + n - 1$, where $n \leq \frac{1}{2}p + q$ is the number of local minima in the spine.  Thus the genus of the common stabilization $(T, H^-_T, H^+_T)$ is at most $\frac{3}{2}p + 2q -1$.
\end{proof}

\bibliographystyle{amsplain}
\bibliography{handlethin}

\end{document}